\documentclass[a4paper,reqno]{amsart}

\usepackage{amsmath,amsthm,amssymb}
\usepackage{graphicx}

\newtheorem{theorem}{Theorem}[section]

\theoremstyle{definition}
\newtheorem{definition}[theorem]{Definition}
\newtheorem{example}[theorem]{Example}

\theoremstyle{remark}
\newtheorem{remark}[theorem]{Remark}

\AtBeginDocument{{\noindent\small
This is a preprint of a paper whose final and definite form is with 
\emph{Chaos, Solitons \& Fractals}, ISSN: 0960-0779. 
Submitted 1 Dec 2016; Article revised 17 Apr 2017;
Article accepted for publication 19 Apr 2017;
DOI: 10.1016/j.chaos.2017.04.034.} 
\vspace{9mm}}

\begin{document}

\markboth{S. Jahanshahi, E. Babolian, D. F. M. Torres and A. Vahidi}{%
Approximating fractional integrals and derivatives}

\title[Approximating fractional integrals and derivatives]{%
A Fractional Gauss--Jacobi quadrature rule\\
for approximating fractional integrals\\
and derivatives\thanks{Part of first author's PhD project.
Partially supported by Islamic Azad University, Tehran, Iran;
and CIDMA-FCT, Aveiro, Portugal, within project UID/MAT/04106/2013.}}


\author[S. Jahanshahi, E. Babolian, D. F. M. Torres and A. Vahidi]{%
S. Jahanshahi, E. Babolian, D. F. M. Torres and A. Vahidi}

\address{Salman Jahanshahi\newline
\indent Department of Mathematics, 
Science and Research Branch,
Islamic Azad University,\newline
\indent Tehran, Iran}
\email{s.jahanshahi@iausr.ac.ir}
   
\address{Esmail Babolian\newline
\indent Department of Mathematics, 
Science and Research Branch,
Islamic Azad University,\newline
\indent Tehran, Iran}
\email{babolian@khu.ac.ir}

\address{Delfim F. M. Torres\newline
\indent Center for Research and Development in Mathematics and Applications (CIDMA),\newline
\indent Department of Mathematics, University of Aveiro, 3810--193 Aveiro, Portugal}
\email{delfim@ua.pt}

\address{Alireza Vahidi\newline
\indent Department of Mathematics, Shahre Rey Branch,
Islamic Azad University,\newline
\indent Tehran, Iran}
\email{alrevahidi@iausr.ac.ir}
 

\begin{abstract}
We introduce an efficient algorithm for computing fractional integrals
and derivatives and apply it for solving problems
of the calculus of variations of fractional order. 
The proposed approximations are particularly useful 
for solving fractional boundary value problems.
As an application, we solve a special class of fractional
Euler--Lagrange equations. The method is based on Hale and Townsend algorithm
for finding the roots and weights of the fractional Gauss--Jacobi quadrature
rule and the predictor-corrector method introduced by Diethelm for solving
fractional differential equations. Illustrative examples show that the
given method is more accurate than the one introduced in
[Comput. Math. Appl. 66 (2013), no.~5, 597--607],
which uses the Golub--Welsch algorithm for evaluating 
fractional directional integrals.
\end{abstract}





\maketitle


\section{Introduction}

Finding numerical approximations of fractional integrals or fractional derivatives
of a given function is one of the most important problems in theory of numerical
fractional calculus. The operators of fractional integration and fractional
differentiation are more complicated than the classical ones, so their evaluation 
is also more difficult than the integer order case. Li et al. use spectral 
approximations for computing the fractional integral and the Liouville--Caputo 
derivative \cite{sal2}. They also developed
numerical algorithms to compute fractional integrals and Liouville--Caputo derivatives
and for solving fractional differential equations based on piecewise polynomial
interpolation \cite{sal3}. In \cite{sal4}, Pooseh et al. presented two approximations
derived from continuous expansions of Riemann--Liouville fractional derivatives
into series involving integer order derivatives and they present application
of such approximations to fractional differential equations and fractional problems
of the calculus of variations. Some other computational algorithms
are also introduced in \cite{sal7,sal5,sal6}.
For increasing the accuracy of the calculation, using the Gauss--Jacobi quadrature rule
is appropriate for removing the singularity of the integrand. So considering
the nodes and weights of the quadrature rule is an important problem.
There are many good papers in the literature addressing the question of how to
find the nodes and weights of the Gauss quadrature rule---see \cite{sal10,sal9,sal8}
and references therein. The more applicable and developed method
is the Golub--Welsch (GW) algorithm \cite{sal12,sal11}, that is used by many
of the mathematicians who work in numerical analysis. This method takes $O(n^2)$
operations to solve the problem of finding the nodes and weights.
Here we use a new method introduced by Hale and  Townsend \cite{sal13},
which is based on the Glasier--Liu--Rokhlin (GLR) algorithm \cite{sal14}.
It computes all the nodes and weights of the $n$-point quadrature rule
in a total of $O(n)$ operations.

The structure of the paper is as follows.
In Section~\ref{sec2}, we introduce the definitions of fractional operators
and some relations between them. Section~\ref{sec3} discusses the Gauss--Jacobi
quadrature rule of fractional order and its application to approximate
the fractional operators. In Section~\ref{sec4} we present two methods
for finding the nodes and weights of Gauss--Jacobi and discuss their advantages
and disadvantages. Two illustrative examples are solved. In Section~\ref{sec5}
applications to ordinary fractional differential equations are presented.
In Section~\ref{sec6} we investigate problems of the calculus of variations
of fractional order and present a new algorithm for solving boundary
value problems of fractional order. We end with Section~\ref{sec7} 
of conclusions and possible directions of future work.


\section{Preliminaries and notations about fractional calculus}
\label{sec2}

In this section we give some necessary preliminaries of the
fractional calculus theory \cite{MR2218073,sal15},
which will be used throughout the paper.

\begin{definition}
The left and right Riemann--Liouville fractional integrals
of order $\alpha$ of a given function $f$ are defined by
\begin{equation*}
_aI_x^\alpha f(x)=\frac{1}{\Gamma(\alpha)}
\int_a^x(x-t)^{\alpha-1}f(t)dt
\end{equation*}
and
\begin{equation*}
_xI_b^{\alpha}f(x)
=\frac{1}{\Gamma(\alpha)}\int_x^b(t-x)^{\alpha-1}f(t)dt,
\end{equation*}
respectively, where $\Gamma$ is Euler's gamma function, that is,
\begin{equation*}
\Gamma(x)=\int_0^\infty t^{x-1}e^{-t}dt,
\end{equation*}
$\alpha>0$ with $n-1<\alpha\leq n$, $n\in\mathbb{N}$, and $a<x<b$.
The left Riemann--Liouville fractional operator has the following properties:
\begin{equation*}
{_aI_x^\alpha} {_aI_x^\beta}
= {_aI_x}^{\alpha+\beta},
\quad {_aI_x^\alpha} {_aI_x^\beta}
= {_aI_x^\beta} {_aI_x^\alpha},
\end{equation*}
\begin{equation*}
_aI_x^\alpha x^\mu=\frac{\Gamma(\mu+1)}{\Gamma(\alpha+\mu+1)}x^{\alpha+\mu},
\end{equation*}
where $\alpha,\beta\geq 0$ and $\mu>-1$. Similar relations hold for the right
Riemann--Liouville fractional operator. On the other hand, we have the left
and right Riemann--Liouville fractional derivatives of order $\alpha>0$ 
that are defined by
\begin{equation*}
_aD_x^{\alpha}f(x)=\frac{1}{\Gamma(n-\alpha)}\frac{d^n}{dx^n}\int_a^x(x-t)^{n-\alpha-1}f(t)dt
\end{equation*}
and
\begin{equation*}
_xD_b^\alpha f(x)=\frac{(-1)^n}{\Gamma(n-\alpha)}\frac{d^n}{dx^n}\int_x^b(t-x)^{n-\alpha-1}f(t)dt,
\end{equation*}
respectively.
\end{definition}

There are some disadvantages when trying to model real world phenomena with fractional
differential equations, when fractional derivatives are taken in Riemann--Liouville sense.
One of them is that the Riemann--Liouville derivative of the constant function 
is not zero. Therefore, a modified definition of the fractional differential operator,
which was first considered by Liouville and many decades later
proposed by Caputo \cite{sal16}, is considered.

\begin{definition}
The left and right fractional differential operators in Liouville--Caputo 
sense are given by
\begin{equation*}
_a^CD_x^\alpha f(x)
=\frac{1}{\Gamma(n-\alpha)}\int_a^x(x-t)^{n-\alpha-1}f^{(n)}(t)dt
\end{equation*}
and
\begin{equation*}
_x^CD_b^\alpha f(x)
=\frac{(-1)^n}{\Gamma(n-\alpha)}\int_x^b(t-x)^{n-\alpha-1}f^{(n)}(t)dt,
\end{equation*}
respectively.
\end{definition}
The Liouville--Caputo derivative has the following two properties
for $n-1<\alpha\leq n$ and $f\in L_1[a,b]$:
\begin{equation*}
({_a^CD_x^\alpha} {_aI_x^\alpha} f)(x)=f(x)
\end{equation*}
and
\begin{equation*}
({_aI_x^\alpha} {_a^CD_x^\alpha} f)(x)
=f(x)-\sum_{k=0}^{n-1}f^{(k)}(0^{+})\frac{(x-a)^k}{k!},\quad t>0.
\end{equation*}

\begin{remark}
Using the linearity property of the ordinary integral operator, one deduces
that left and right Riemann--Liouville integrals, left and right
Riemann--Liouville derivatives and left and right
Liouville--Caputo derivatives are linear operators.
\end{remark}

Another definition of a fractional differential operator, that is useful for numerical
approximations, is the Gr\"{u}nwald--Letnikov derivative, which is a generalization 
of the ordinary derivative. It is defined as follows:
\begin{equation*}
D_{GL}^{\alpha}=\lim_{n\to\infty}\frac{(\frac{t}{N})^{-\alpha}}{\Gamma(-\alpha)}
\sum_{j=0}^{n-1}\frac{\Gamma(j-\alpha)}{\Gamma(j+1)}f\left(t-\frac{tj}{n}\right).
\end{equation*}


\section{Fractional Gauss--Jacobi quadrature rule}
\label{sec3}

It is well known that the Jacobi polynomials
$\{P_n^{(\lambda,\nu)}(x)\}_{n=0}^{\infty}$, $\lambda, \nu >-1$, $x\in [-1,1]$,
are the orthogonal system of polynomials
with respect to the weight function
$$
(1-x)^\lambda (1+x)^\nu,
\quad \lambda, \nu >-1,
$$
on the segment $[-1,1]$:
\begin{equation}
\label{eq3.1}
\int_{-1}^{1}(1-x)^\lambda (1+x)^\nu P_n^{(\lambda,\nu)}(x)
P_m^{(\lambda,\nu)}(x) dx=\partial_n^{\lambda,\nu}\delta_{mn},
\end{equation}
where
\begin{equation*}
\partial_n^{\lambda,\nu}=\Vert P_n^{(\lambda,\nu)}\Vert_{w^{\lambda,\nu}}^2,
\quad w^{\lambda,\nu}(x)=(1-x)^\lambda (1+x)^\nu,
\end{equation*}
$$
\delta_{mn}=\left\{
\begin{array}{ll}
1, & m=n,\\
0, & m\neq n
\end{array} \right.
$$
(see, e.g., \cite{sal17}). These polynomials satisfy
the three-term recurrence relation
\begin{equation}
\label{eq3.4}
\begin{gathered}
P_0^{(\lambda,\nu)}(x)=1,
\quad P_1^{(\lambda,\nu)}(x)=\frac{1}{2} (\lambda+\nu+2) x +\frac{1}{2} (\lambda-\nu),\\
P_{n+1}^{(\lambda,\nu)}(x)=\left( a_n^{\lambda,\nu} x- b_n^{\lambda,\nu}\right)
P_n^{(\lambda,\nu)}(x) - c_n^{\lambda,\nu} P_{n-1}^{(\lambda,\nu)}(x), 
\quad n\geq 2,
\end{gathered}
\end{equation}
where
\begin{align*}
a_n^{\lambda,\nu}
&=\dfrac{(2n+\lambda+\nu+1) (2n+\lambda+\nu+2)}{2(n+1) (n+\lambda+\nu+1)},\\
&\quad\\
b_n^{\lambda,\nu}
&=\dfrac{(\nu^2-\lambda^2) (2n+\lambda+\nu+1)}{2(n+1) 
(n+\lambda+\nu+1) (2n+\lambda+\nu)},\\
&\quad\\
c_n^{\lambda,\nu}
&=\dfrac{(n+\lambda) (n+\nu) (2n+\lambda+\nu+2)}{(n+1) 
(n+\lambda +\nu +1) (2n+\lambda +\nu)}.
\end{align*}
The explicit form of the Jacobi polynomials is
\begin{equation}
\label{eq3.5}
P_n^{(\lambda,\nu)}(x)=\sum_{k=0}^{n}
\frac{2^{n-k}n!(k+\nu+1)_{n-k}}{(n-k)!(n+k+\lambda+\nu+1)_{n-k}k!}(t-1)^k,
\end{equation}
where we use Pchhammer's notation:
$$
(a)_l=a(a+1)(a+2)\cdots (a+l-1)
$$
(see \cite{sal25}). Furthermore, the Jacobi polynomials
satisfy in the following relations:
\begin{equation*}
P_n^{(\lambda,\nu)}(-x)=(-1)^n  P_n^{(\lambda,\nu)}(x),
\end{equation*}
\begin{equation*}
\frac{d}{dx}P_n^{(\lambda,\nu)} (-x)=(-1)^n \frac{d}{dx}P_n^{(\lambda,\nu)}(x),
\end{equation*}
\begin{multline}
\label{equ3.4}
(2n+\lambda+\nu) (1-x^2) \frac{d}{dx} P_n^{(\lambda,\nu)}(x)\\
=n\left( \lambda-\nu-(2n+\lambda+\nu)x\right) P_n^{(\lambda,\nu)}(x)
+2(n+\lambda) (n+\nu)P_{n-1}^{(\lambda,\nu)}(x),
\end{multline}
\begin{equation*}
P_n^{(\lambda,\nu)}(1)=\binom{n+\lambda}{n},
\quad P_n^{(\lambda,\nu)}(-1)
=(-1)^j \,\dfrac{\Gamma (n+\nu+1)}{n!\Gamma(\nu+1)}.
\end{equation*}
On the other hand, the Gauss--Jacobi quadrature
formula with remainder term is given by
\begin{equation}
\label{eq:GJqf}
\int_{-1}^1 (1-x)^\lambda (1+x)^\nu f(x) dx
= \sum_{k=1}^n A_k f(x_k)+R_n(f),
\end{equation}
where $x_k$, $k=1,\ldots, n$, are the zeros of
$P_n^{(\lambda,\nu)}$, the weights $A_k$ are given by
\begin{equation*}
A_k = \dfrac{2^{\lambda+\nu+1} \Gamma(\lambda+n+1)
\Gamma(\nu+n+1)}{n! \Gamma (\lambda+\nu+n+1)
\left(\frac{d}{dx} P_n^{(\lambda,\nu)}(x_k)\right)^2 (1-x_k^2)}
\end{equation*}
and the error is given by
\begin{equation*}
R_n(f) = \dfrac{f^{(2n)}(\eta)}{(2n)!}
\,\dfrac{2^{\lambda+\nu+2n+1} n! \Gamma(\lambda+\eta+1)
\Gamma(\nu +\eta +1)\Gamma(\nu+\lambda+\eta
+1)}{(\lambda+\nu+2n+1)\left(\Gamma(\lambda+\nu+2n+1)\right)^2}
\end{equation*}
(see \cite{sal17}). We know that formula \eqref{eq:GJqf}
is exact for all polynomials of degree up to $2n-1$
and is valid if $f(x)$ possesses no singularity in $[-1,1]$
except at points $\pm 1$.

Now, as a special case, we introduce the fractional Jacobi polynomials
$f_n^{\alpha-1}(x)$, that is, $\lambda=\alpha-1$ and $\nu=0$ in \eqref{eq3.1}.
The set of fractional Jacobi polynomials
$\{f_n^{\alpha-1}\}_{n=0}^{\infty}$
is an orthogonal system of polynomials that are orthogonal
with respect to the weight function $(1-x)^{\alpha-1}$,
$\alpha>0$, on the segment $[-1,1]$. It means that
\begin{equation*}
\int_{-1}^1w^{\alpha-1}(x) f_n^{\alpha-1}(x) f_m^{\alpha-1}(x) dx
= \zeta_n^{\alpha-1} \delta_{mn},
\end{equation*}
where
\begin{equation*}
\zeta_n^{\alpha-1}
=\Vert f_n^{\alpha-1}\Vert_{w^{\alpha-1}}^2,
\quad w^{\alpha-1}(x)= (1-x)^{\alpha-1}.
\end{equation*}
The three-term recurrence relation for fractional Jacobi polynomial is given by
\begin{gather*}
f_0^{\alpha-1}(x)=1,
\quad f_1^{\alpha-1}(x)=\frac{\alpha-1}{2} (x+1),\\
f_{n+1}^{\alpha-1} (x)= \left(A_n^\alpha x+ B_n^\alpha \right)
f_n^{\alpha-1}-C_n^\alpha f_{n-1}^{\alpha-1}(x),\quad n\geq 2,
\end{gather*}
where
\begin{align*}
A_n^\alpha
&= \dfrac{(2n+\alpha) (2n+\alpha+1)}{2(n+1) (n+\alpha)},\\
&\quad \\
B_n^\alpha
&= \dfrac{(\alpha-1)^2 (2n+\alpha)}{2(n+1)(2n+\alpha-1) (n+\alpha)},\\
&\quad \\
C_n^\alpha
&=\dfrac{n(n+\alpha-1) (2n +\alpha +1)}{(n+1) (n+\alpha) (2n+\alpha-1)}.
\end{align*}
Using \eqref{equ3.4}, we have that
\begin{multline}
\label{equ3.5}
(2n+1-\alpha) (1-x^2) \frac{d}{dx} f_n^{\alpha -1} (x)\\
=n\left(1-\alpha-(2n+1-\alpha)x\right) f_n^{\alpha -1}(x)
+2n(n+1-\alpha) f_{n-1}^{\alpha -1}(x).
\end{multline}
Furthermore, the fractional Gauss--Jacobi quadrature rule is
\begin{equation}
\label{eq3.19}
\int_{-1}^1 (1-x)^{\alpha-1} f(x) dx = \sum_{k=1}^n l_k f(x_k) + E_n^\alpha (f),
\end{equation}
where $x_k$, $k=1,2,\ldots, n$, are the zeros
of $f_n^{\alpha-1}$, the weights $l_k$ are given by
\begin{equation*}
l_k=\dfrac{2^\alpha}{(1-x_k^2)[\frac{d}{dx}f_n^{\alpha-1} (x_k)]^2}
\end{equation*}
and
\begin{equation*}
E_n^\alpha (f)=\dfrac{2^n f^{(2n)}(\eta)}{(2n)! (2n+\alpha)}
\, \left(\dfrac{2^n n! \Gamma (n+\alpha)}{\Gamma (2n+\alpha)}\right)^2.
\end{equation*}
Consider the left Riemann--Liouville fractional integral
of order $\alpha>0$ for $x>0$:
\begin{equation}
\label{eq3.22}
_0I_x^\alpha f(x)=\dfrac{1}{\Gamma (\alpha)}
\int_0^x (x-t)^{\alpha-1} f(t) dt.
\end{equation}
For calculating the above integral at the collocation nodes,
first we transform the interval $[0,x]$ into the segment
$[-1,1]$ using the changing of variable
$$
s=2\left(\frac{t}{x}\right)-1,
\quad ds=\frac{2}{x} dt.
$$
Therefore, we can rewrite \eqref{eq3.22} as
\begin{equation*}
\dfrac{1}{\Gamma (\alpha)} \int_0^x (x-t)^{\alpha-1} f(t) dt
=\left(\frac{x}{2}\right)^\alpha \dfrac{1}{\Gamma(\alpha)}
\int_{-1}^1 \dfrac{\rho (s) ds }{(1-s)^{1-\alpha}},
\end{equation*}
where
$$
\rho(s)=f\left(\frac{x}{2} (s+1)\right).
$$
So, we have
\begin{equation}
\label{eq3.24}
_0I_x^\alpha f(x)=\dfrac{1}{\Gamma (\alpha)}\left(\frac{x}{2}\right)^\alpha
\int_{-1}^1 (1-s)^{\alpha-1} \rho(s) ds,
\end{equation}
and using the $n$ point fractional Gauss--Jacobi quadrature rule
\eqref{eq3.19} for \eqref{eq3.24} we get
\begin{equation}
\label{eq3.25}
_0I_{s_i}^\alpha f (x) = \dfrac{(\frac{s_i}{2})^\alpha}{\Gamma (\alpha)}
\sum_{k=1}^n l_k f \left( \frac{s_i}{2} (x_k+1)\right)+ E_n^\alpha (f),
\end{equation}
where
\begin{equation}
\label{eq3.26}
l_k=\dfrac{2^\alpha}{(1-x_k)^2 \left(\frac{d}{dx} f_n^{\alpha-1} (x_k)^2\right)},
\end{equation}
\begin{equation*}
E_n^\alpha(f) = \left(\frac{s_i}{2}\right)^\alpha \frac{1}{\Gamma (\alpha)}
\, \dfrac{f^{(2n)} \left(\frac{s_i}{2}(\eta+1)\right) }{(2n)!}
\, \dfrac{2^{2n} (n!)^2}{(2n+\alpha)} \, \left(
\dfrac{\Gamma (n+\alpha)}{\Gamma (2n+\alpha)}\right)^2
\end{equation*}
and $x_k$, $k=1, \ldots, n$, are quadrature nodes. In this way
we just need to compute the nodes and weights of the
fractional Gauss--Jacobi quadrature rule using a fast and accurate algorithm.
Similarly, an approximation formula for computing the left Liouville--Caputo 
fractional derivative
of a smooth function $f$, using the suggested method, is given by
\begin{equation}
\label{eq3.28}
{_0^CD_{s_i}^\alpha} f(x)=\dfrac{(\frac{s_i}{2})^{1-\alpha}}{\Gamma(1-\alpha)}
\sum_{k=1}^n l_k f' \left( \frac{s_i}{2} (x_k+1)\right)+ E_n^\alpha (f),
\quad \alpha \in (0,1),
\end{equation}
where $x_k$, $k=1,2,\ldots,n$, are the zeros of $f_n^{-\alpha}$,
$$
l_k=\dfrac{2^{1-\alpha}}{(1-x_k)^2 \left(\frac{d}{dx} f_n^{-\alpha} (x_k)^2\right)}
$$
and
$$
E_n^\alpha (f) = \dfrac{(\frac{s_i}{2})^{1-\alpha}}{\Gamma(1-\alpha)}
\dfrac{f^{(2n+1)}\left(\frac{s_i}{2} (\eta +1)\right)}{(2n+1)!}
\, \dfrac{2^{2n} (n!)^2}{(2n+1-\alpha)}
\, \left(\dfrac{\Gamma (n+1-\alpha)}{\Gamma (2n+1-\alpha)}\right)^2.
$$
Analogous formulas hold for the right Liouville--Caputo fractional derivative.
Next we present two methods for computing the nodes and weights 
of the fractional Gauss--Jacobi quadrature rule \eqref{eq3.19}.
The first one is based on Golub--Welsch algorithm, while the second
is a recent method introduced by Hale and Townsend \cite{sal13},
which is based on Newton's iteration for finding roots
and a good asymptotic formula for the weights.


\section{Two methods for calculating nodes and weights}
\label{sec4}

Pang et al. \cite{sal1} use Gauss--Jacobi and Gauss--Jacobi--Lobatto quadrature
rules for computing fractional directional integrals, together with 
the Golub--Welsch (GW) algorithm for computing nodes and weights 
of quadrature rules. The GW algorithm exploits the three-term recurrence 
relations \eqref{eq3.4} satisfied by all real orthogonal polynomials.
This relation gives rise to a symmetric tridiagonal matrix,
\begin{equation*}
x\left[
\begin{array}{c}
{J}_0^{\lambda,\nu} (x) \\
{J}_1^{\lambda,\nu} (x) \\
\vdots \\
{J}_{N-2}^{\lambda,\nu} (x)\\
{J}_{N-1}^{\lambda,\nu} (x)
\end{array}\right]
=\left[ \begin{array}{ccccc}
A_1 & B_1 & 0 & \cdots & 0\\
B_1 & A_2 & B_2 & \cdots & 0 \\
\vdots & \ddots & \ddots & \ddots & \vdots \\
0 & \cdots & \ddots & A_{N-1} & B_{N-1}\\
0 & \cdots & 0  & B_{N-1} & A_N\\
\end{array}
\right] \,
\left[ \begin{array}{c}
{J}_0^{\lambda,\nu} (x)\\
{J}_1^{\lambda,\nu} (x)\\
\vdots \\
{J}_{N-2}^{\lambda,\nu} (x)\\
{J}_{N-1}^{\lambda,\nu} (x)
\end{array}\right]+
\left[  \begin{array}{c}
0\\
0\\
\vdots\\
0\\
B_N{J}_{N}^{\lambda,\nu} (x)
\end{array}\right] ,
\end{equation*}
where
\begin{equation*}
A_n=\frac{-b_n}{a_n},\quad B_n=\sqrt{\frac{c_{n+1}}{a_n a_{n+1}}},
\quad {J}_{n}^{\lambda,\nu}(x)=\dfrac{{J}_{n}^{\lambda,\nu}}{\sqrt{c_{n,\lambda,\nu}}},
\end{equation*}
in which
\begin{equation*}
a_n = \left\{
\begin{array}{ll}
\dfrac{(2n+\lambda+\nu-1) (2n+\lambda+\nu)}{2n (n+\lambda+\nu)},
& \lambda+\nu \neq -1,\\
\dfrac{2n-1}{n}, &  n\geq 2, \lambda+\nu =-1,\\
\dfrac{1}{2}, & n=1, \lambda+\nu=-1,
\end{array} \right.
\end{equation*}
\begin{equation*}
b_n = \left\{
\begin{array}{ll}
-\dfrac{(2n+\lambda+\nu+1) (\lambda^2-\nu^2)}{2n (n+\lambda+\nu) (2n+\lambda+\nu -2)},
& \lambda+\nu \neq -1,\\
\,&\,\\
\dfrac{\nu^2-\lambda^2}{n(2n-3)}, & n\geq 2, \lambda+\nu=-1,\\
\,&\,\\
\dfrac{\nu^2-\lambda^2}{2}, & n=1, \lambda+\nu=-1,
\end{array} \right.
\end{equation*}
\begin{equation*}
c_n=\dfrac{(n+\lambda-1) (n+\nu-1) (2n+\lambda+\nu)}{n(n+\lambda+\nu) 
(2n+\lambda+\nu-2)}, \, n\geq2,
\end{equation*}
$$
c_{n,\lambda,\nu}=\partial^{\lambda,\nu}.
$$
It is easy to prove that $x_i$ is the zero of
$P_n^{(\lambda,\nu)}$ if and only if $x_i$
is the eigenvalue of the tridiagonal matrix \cite{sal1}.
Moreover, Golub and Welcsh proved that the weights of quadrature
are the first component of corresponding eigenvectors \cite{sal11}.
This algorithm takes $O(n^2)$ operations to solve the eigenvalue problem
by taking advantage of the structure of the matrix and noting that only
the first component of the normalized eigenvector need to be computed.
An alternative approach for computing the nodes and weights of the Gauss--Jacobi rule
is to use the same three-term recurrence relation in order to compute Newton's iterates,
which converge to the zeros of the orthogonal polynomial \cite{sal18}.
Since the recurrence requires $O(n)$ operations for each evaluation of the polynomial
and its derivative, we have the total complexity of order $O(n^2)$ 
for all nodes and weights.
Furthermore, the relative maximum error in the weights is of order $O(n)$
and for the nodes is independent of $n$. Here we develop the new technique
introduced by Hale and Townsend, which utilizes asymptotic formulas
for both accurate initial guesses of the roots and efficient evaluation
of the fractional Jacobi polynomial $f_n^{\alpha-1}$ inside Newton's method \cite{sal13}.
With this method it is possible to compute the nodes and weights
of the fractional Gauss--Jacobi rule in just $O(n)$ operations
and almost full double precision of order $O(1)$. For performance of the method,
first put $\theta_k=\cos^{-1}x_k$ to avoid the existing clustering in computing
$w_k$, because of presence of the term $(1-x_k^2)$ in the denominator of
\eqref{eq3.26} that lead to cancellation error for $x_k = \pm 1$. 
Then, using a simple method, like the bisection one, 
for finding $\theta_k^{[0]}$ as the initial guess of the $k$th root,
we construct the Newton iterates for finding the nodes as follows:
\begin{equation*}
\theta_k^{[j+1]}=\theta_k^{[j]}-\frac{f_n^{\alpha-1}\left(
\cos\theta_k^{[j]}\right)}{\left[-\sin \theta_k^{[j]}
\frac{d}{d\theta}f_n^{\alpha-1} (\cos \theta_k^{[j]}) \right]},
\quad j=0,1,2,\ldots
\end{equation*}
Once the iterates have converged, the nodes are given by
$x_k=\cos\theta_k$ and using \eqref{eq3.26} the weights are given by
\begin{equation*}
l_k=\left( \frac{d}{d\theta} f_n (\cos \theta_k)\right)^{-2}.
\end{equation*}
Since the zeros of the orthogonal fractional Jacobi polynomial are simple \cite{sal17},
we conclude that the Newton iterates converge quadratically \cite{sal22}. Furthermore,
since all fractional Jacobi polynomials satisfy the relation
\eqref{eq3.5}, we just need to consider all calculations for
$x\in [0,1]$, i.e., $\theta \in [0,\frac{\pi}{2}]$. There are three asymptotic
formulas for finding the nodes.
\begin{enumerate}
\item[1)] One is given by Gatteschi and Pittaluga \cite{sal19}
for Jacobi polynomials and here to the case of fractional Jacobi polynomials:
\begin{equation*}
\tilde{x}_k=\cos\left[f_k+\frac{1}{4\rho^2} \left( \frac{1}{4}
- (\alpha-1)^2\right) \cot\left(\frac{f_k}{2}\right)
-\frac{1}{4}\tan\frac{f_k}{2}\right]+O(n^{-4}),
\end{equation*}
$|\alpha|\leq \frac{1}{2}$, where $\rho=n+\frac{\alpha}{2}$ and
$f_k=(k+\frac{\alpha}{2}-\frac{3}{4})\frac{\pi}{\rho}$.

\item[2)] Let $j_{\alpha-1,k}$ denote the $k$th root of Bessel's function
$J_{\alpha-1}(z)$. Then the approximation given by Gatteschi and Pittaluga \cite{sal19}
for the nodes near $x=1$ becomes
\begin{gather*}
\tilde{x_k}=\cos \left[\dfrac{j_{\alpha-1,k}}{\gamma}\left(1
-\dfrac{4-(\alpha-1)^2}{720 \gamma^4}\right)\left(\frac{ j_{\alpha-1, k}^2}{2}
+\alpha^2 -2\alpha\right)\right]\\
+j_{\alpha-1,k}^5 O(n^{-7}),
\end{gather*}
$\alpha \in [-\frac{1}{2}, \frac{1}{2}]$, where
$\gamma=\sqrt{\rho^2+(2-\alpha^2-2\alpha)/12}$.

\item[3)] Other asymptotic formulas for the nodes near
$x=1$ are given by Olver et al. \cite{sal20}:
\begin{gather*}
\tilde{x}_k=\cos \left[\psi_k+ (\alpha^2-2\alpha+3/4)\,
\dfrac{\psi_k \cot (\psi_k)-1}{2\rho^2 \psi_k}
- (\alpha-1)^2 \dfrac{\tan (\frac{\psi_k}{2})}{4\rho^2}\right]\\
+ j_{\alpha-1 ,k}^2 O (n^5),
\end{gather*}
$\alpha>-\frac{1}{2}$, where $\psi_k=\dfrac{j_{\alpha-1 ,k}}{\rho}$.
\end{enumerate}
For computing the weights $l_k$ we just need to evaluate
$\dfrac{d}{d\theta} f_n^{\alpha-1} (\cos \theta)$
in the $\theta_k$. To this end, we use relation \eqref{equ3.5}
between $f_n^{\alpha-1}$ and $(f_n^{\alpha-1})'$.
So we just need to compute the value of $ f_n^{\alpha-1}$ at
$x = \cos\theta$. This work is done by using an asymptotic formula introduced in
\cite{sal13}, which takes the following form for fractional Jacobi polynomials:
\begin{multline*}
\sin^{\alpha-1/2}(\theta/2)\cos^{1/2}(\theta/2)f_n^{\alpha-1} (\cos \theta)\\
=\dfrac{2^{2\rho} B (n+\alpha,n+1)}{\pi}\sum_{m=0}^{M-1}
\dfrac{f_m(\theta)}{2^m (2\rho+1)_m}+V_{M,n}^{\alpha-1}(\theta),
\end{multline*}
where $\rho=n+\alpha/2$, $B(\alpha,\beta)$ is the Beta function,
\begin{equation*}
f_m(\theta)=\sum_{l=0}^m \dfrac{c_{m,l}^{\alpha-1}}{l! (m-l)!}
\,\dfrac{\cos (\theta_{n,m,l})}{\sin^l (\theta/2) \cos ^{m-l} (\theta/2)},
\end{equation*}
\begin{equation*}
\theta_{n,m,l}=\frac{1}{2} (2\rho+m)\theta -\frac{1}{2}(\alpha+l-a/2)\pi
\end{equation*}
and
\begin{equation*}
C_{m,l}^\alpha= \left(\alpha-\frac{1}{2}\right)_l \left(\frac{3}{2}
-\alpha\right)_l \left(\frac{1}{4}\right)_{m-l},
\end{equation*}
where for $\alpha\in(-1/2,1/2)$ the error term $V_{M,n}$
is less than twice the magnitude of the first neglected term.
We give two examples illustrating the usefulness of the method.

\begin{example}
\label{exa4.1}
The fractional integral of the function $f(t)=\sin(t)$ is defined as
\begin{equation*}
{_0I_t}^\alpha \sin(t)=\frac{1}{\Gamma(\alpha)}\int_{0}^{t}(t-x)^{\alpha-1} \sin(x)dx,
\quad t\in[0,2\pi],\quad \alpha\in(0,1).
\end{equation*}
The explicit expression of the integral, presented in Table 9.1 of \cite{sal21}, is
\begin{equation*}
{_0I_t}^\alpha \sin(t)=\frac{t^\alpha}{2i\Gamma(\alpha+1)}[_1F_1(1;\alpha+1;it)
- {_1F_1}(1;\alpha+1;-it)],
\end{equation*}
where $i=\sqrt{-1}$ and $_1F_1$ is the generalized hyper-geometric function defined by
\begin{equation*}
_1F_1(p;q;z)=\sum_{k=0}^{\infty}\frac{(p)_k z}{(q)_k}.
\end{equation*}
Choosing the test points in the set $\{s_k=\frac{k\pi}{8}, k=0,1,\ldots,16\}$,
the error is given by
\begin{equation}
\label{eq5.4}
\sqrt{\frac{\sum_{k=0}^{16} (E_k-A_k)^2}{\sum_{k=0}^{16} E_k}},
\end{equation}
where $E_k$ is the exact result at $s_k$ and $A_k$ is the approximated result by our method.
The values of \eqref{eq5.4} for $\alpha=0.25$, $0.5$ and $0.75$ are listed in the Table~\ref{Table 1}.
We used the command \texttt{jacpts} of the \texttt{chebfun} software \cite{sal13}
for finding the nodes and weights of the fractional Gauss--Jacobi quadrature (\texttt{fgjq}) rule.
The errors show that the method is more accurate than the method introduced in \cite{sal1}.
\begin{table}[ht]
\caption{The errors obtained for Example \ref{exa4.1} from \eqref{eq5.4},
with $\alpha= 0.25, 0.5, 0.75$ and  $n=5,6,7,8,16$ in \eqref{eq3.25}.}
\label{Table 1}
\begin{tabular}{|c|c|c|c|c|c|c|c|c|c|c|c|}
\hline
$\alpha$ & $n=5$&$n=6$&$n=7$&$n=8$&$n=16$\\\hline
$0.25$ & $3.22\times10^{-6}$&$5.14\times10^{-8}$
&$6.1\times10^{-10}$& $5.58\times10^{-12}$&$2.81\times10^{-15}$\\
$0.5$ & $4.85\times10^{-6}$ & $7.75\times10^{-8}$ & $9.18\times10^{-10}$
& $8.37\times10^{-12}$ & $7.12\times10^{-16}$ \\
$0.75$ & $5.35\times10^{-6}$ & $8.35\times10^{-8}$
& $9.65\times10^{-10}$ & $8.6\times10^{-12}$ & $1.39\times10^{-15}$ \\\hline
\end{tabular}
\end{table}
\end{example}

\begin{example}
\label{exam:5.2.}
In this example we consider the function $f(t)=t^4$. The Liouville--Caputo
fractional derivative of $f$ of order $0<\alpha<1$ is defined as
\begin{equation*}
_0^CD_t ^{\alpha}t^4=\frac{1}{\Gamma(1-\alpha)}
\int_{0}^{t}4(t-x)^{1-\alpha}t^3dx,\quad \alpha\in(0,1).
\end{equation*}
The explicit form of the result is
\begin{equation*}
_0^CD_t ^{\alpha} t^4=\frac{\Gamma(5)}{\Gamma(4.5)}t^{3.5}
\end{equation*}
(see \cite{sal4}). Using \eqref{eq3.28}, we can write
\begin{equation*}
_0^CD_{s_i} ^{\alpha}t^4\simeq \dfrac{(\frac{s_i}{2})^{1-\alpha}}{\Gamma (1-\alpha)}
\sum_{k=1}^n 4 l_k \left(\frac{s_i}{2} (x_k+1)\right)^3.
\end{equation*}
Here the nodes and weights coincide with the nodes and weights of Example~\ref{exa4.1}
for $\alpha=0.5$. With the test point set $\{s_k=\frac{k}{10}, k=0,1,2,\ldots,10\}$
the error defined by \eqref{eq5.4} is listed in Table~\ref{Table 2}. Looking to Table~\ref{Table 2},
one can see that the best approximation is obtained for $n=2$
and that by increasing $n$ we do not get better results.
\begin{table}[ht]
\caption{The errors obtained for Example~\ref{exam:5.2.} from \eqref{eq5.4}
with $n=2,3,\ldots,7$ in \eqref{eq3.28} for approximating the Liouville--Caputo 
fractional derivative of order $\alpha= 0.5$ of function $f(t)=t^4$.}
\label{Table 2}
\begin{tabular}{|c|c|c|c|c|c|c|c|c|c|c|} \hline
$n=2$&$n=3$&$n=4$&$n=5$&$n=6$&$n=7$\\\hline
$4.0\times10^{-17}$ & $8.8\times10^{-16}$ & $1.9\times10^{-16}$
& $6.4\times10^{-16}$ & $3.8\times10^{-16}$ & $3.8\times10^{-16}$\\\hline
\end{tabular}
\end{table}
\end{example}


\section{Application to fractional ordinary differential equations}
\label{sec5}

Consider the fractional initial value problem
\begin{equation}
\label{eq6.1}
\left\{
\begin{array}{ll}
_0^CD_t^\alpha y(t)=g(t,y(t)),
& n-1<\alpha<n,\quad n\in\mathbb{N},\quad 0<t<b,\\
y^{(k)}(0)=b_k, & k=0,1,\ldots,n-1,
\end{array} \right.
\end{equation}
where $y(t)$ is the unknown function, $g:\mathbb{R}^2\to\mathbb{R}$
is given and $b_k$, $k=0,1,\ldots,n-1$, are real constants.
The following result establishes existence and uniqueness
of solution for \eqref{eq6.1}.

\begin{theorem}[See \cite{MR2218073,sal15}]
\label{thm:6.1}
Consider the fractional differential equation \eqref{eq6.1}.
Let $K>0$, $h^*>0$, $b_0,b_1,\ldots,b_{n-1}\in\mathbb{R}$,
$G:=[0,h^*]\times[b_0-K,b_0+K]$,
and function $g:G\to\mathbb{R}$ be continuous.
Then, there exists $h > 0$ and a function
$y\in C[0, h]$ solving the Liouville--Caputo fractional initial
value problem \eqref{eq6.1}. If $0<\alpha<1$,
then the parameter $h$ is given by
$$
h:=\min\left\{h^*,\left(\frac{K\Gamma(\alpha+1)}{M}\right)^{1/\alpha}\right\},
\quad M:=\sup_{(t,z)\in G}|g(t,z)|.
$$
Furthermore, if $g$ fulfils a Lipschitz condition with respect to the second variable,
that is,
$$
|g(t,y_1)-g(t,y_2)|\leq L(y_1-y_2)
$$
for some constant $L > 0$ independent of $t, y_1$ and $y_2$,
then the function $y\in C[0, h]$ is unique.
\end{theorem}

In order to solve problem \eqref{eq6.1} using our method, we need the following theorem.

\begin{theorem}
\label{thm:6.2}
Under the assumptions of Theorem~\ref{thm:6.1},
function $y\in C[0, h]$ is a solution to the Liouville--Caputo
fractional differential equation \eqref{eq6.1}
if and only if it is a solution to the Volterra
integral equation of second kind
\begin{equation*}
y(t)=\sum_{k=0}^{n-1}\frac{t^k}{k!}b_k
+\frac{1}{\Gamma(\alpha)}\int_0^t(t-x)^{\alpha-1}g(x,y(x))dx.
\end{equation*}
\end{theorem}

\begin{proof}
Using the Laplace transform formula
for the Liouville--Caputo fractional derivative,
\begin{equation*}
\mathcal{L}\{{_0^CD_t^\alpha} y(t)\}
= s^\alpha Y(s)-\sum_{k=0}^{n-1} s^{\alpha-k-1}y^{(k)}(0)
\end{equation*}
(see \cite{MR2218073,sal15}). Thus, using \eqref{eq6.1}, we can write that
\begin{equation*}
s^\alpha Y(s)-\sum_{k=0}^{n-1}s^{\alpha-k-1}y^{(k)}(0)=G(s,y(s))
\end{equation*}
or
\begin{equation*}
Y(s)=s^{-\alpha}G(s,y(s))+\sum_{k=0}^{n-1}s^{-k-1}b_k,
\end{equation*}
where $G(s,y(s))= \mathcal{L}\{g(t,y(t))\}$. 
Applying the inverse Laplace transform gives
$$
y(t)=\sum_{k=0}^{n-1}\frac{t^k}{k!}b_k
+\frac{1}{\Gamma(\alpha)}\int_0^t(t-x)^{\alpha-1}g(x,y(x))dx,
$$
where we used the relations
\begin{equation*}
\mathcal{L}\{_0I_t^\alpha y(t)\} = \mathcal{L}\left\{\frac{1}{\Gamma(\alpha)}
\int_0^t(t-x)^{\alpha-1}y(x)dx\right\}
=\mathcal{L}\left\{\frac{t^{\alpha-1}}{\Gamma(\alpha)}
*y(t)\right\}=\frac{Y(s)}{s^\alpha}
\end{equation*}
and
$$
\mathcal{L}\{t^{\alpha-1}\}=\frac{\Gamma(\alpha)}{s^\alpha}.
$$
The proof is complete.
\end{proof}

Now, using \eqref{eq3.24}, \eqref{eq3.25} and Theorem~\ref{thm:6.2}, we have
\begin{equation*}
\begin{split}
y(t)&=\sum_{k=0}^{n-1}\frac{t^k}{k!}b_k
+\frac{1}{\Gamma(\alpha)}\int_0^t(t-x)^{\alpha-1}g(x,y(x))dx\\
&=\sum_{k=0}^{n-1}\frac{t^k}{k!}b_k+\frac{1}{\Gamma(\alpha)}\left(
\frac{t}{2}\right)^\alpha\int_{-1}^1(1-s)^{\alpha-1}\tilde{g} (s,\tilde{y}(s))ds,
\end{split}
\end{equation*}
where
$$
\tilde{g}(s,\tilde{y}(s))=g\left(\frac{t}{2}(1+s),
y\left(\frac{t}{2}(1+s)\right)\right),
\quad -1\leq s\leq 1;
$$
$$
\tilde{y}(s)=y\left(\frac{t}{2}(1+s)\right),\quad -1\leq s\leq 1.
$$
So, using the fractional Gauss--Jacobi quadrature rule introduced in Section~\ref{sec2},
we can approximate $y(t)$ by the following formula:
\begin{equation}
\label{eq6.7}
y(t)\simeq\sum_{k=0}^{n-1}\frac{t^k}{k!}b_k
+\frac{1}{\Gamma(\alpha)}\left(\frac{t}{2}\right)^\alpha
\sum_{k=1}^N l_k \tilde{g}(x_k,\tilde{y}(x_k)),
\end{equation}
where $l_k$ and $x_k$ are the weights and nodes of our quadrature rule, respectively.
Also, let $t\in[0,1]$ and $t_j=\frac{j}{n}$ for $j=0,1,\ldots,n$. Furthermore,
let the numerical values of $y(t)$ at $t_0,t_1,\ldots,t_n$ have been obtained,
and let these values be $y_0,y_1,\ldots,y_n$, respectively, with $y_0=b_0$.
Now we are going to compute the value of $y(t)$ at $t_{n+1}$, i.e., $y_{n+1}$.
To this end we use the predictor-corrector method introduced
by Diethelm et al. in \cite{sal26}, but where for the prediction part we use
interpolation at equispaced nodes instead of piecewise linear interpolation.
Using first \eqref{eq6.7}, we know that
\begin{equation}
\label{eq10}
y(t_{n+1})\simeq\sum_{k=0}^{n-1}\frac{t_{n+1}^k}{k!}b_k
+\frac{1}{\Gamma(\alpha)}\left(\frac{t_{n+1}}{2}\right)^\alpha\sum_{k=1}^N
l_k g\left(\frac{t_{n+1}}{2}(1+x_k),y\left(\frac{t_{n+1}}{2}(1+x_k)\right)\right),
\end{equation}
where $l_k$ and $x_k$, $k=1,2,\ldots,N$, are the weights and nodes of
the fractional Gauss--Jacobi quadrature rule introduced in Section~\ref{sec3}.
Looking to \eqref{eq10} we see that for calculating the value of $y(t_{n+1})$
we need to know the value of $g$ at point
$\left(\frac{t_{n+1}}{2}(1+x_k),y\left(\frac{t_{n+1}}{2}(1+x_k)\right)\right)$.
For obtaining these values we act as follows.
First, using the predictor-corrector method introduced in \cite{sal28},
we calculate the values of $y_1(t_i)$, $i=0,1,\ldots,n+1$,
using the following formulas:
\begin{equation*}
y_1(t_{n+1})= \left\{
\begin{array}{ll}
y_0+\frac{k^\alpha}{\Gamma(\alpha+2)}\left(g(t_1,y^{pr}(t_1))
+\alpha g(t_0,y_0)\right), & n=0,\\
y_0+\frac{k^\alpha}{\Gamma(\alpha+2)}\left(g(t_{n+1},
y^{pr}(t_{n+1}))+(2^{\alpha+1}-2) g(t_n,y(t_n))\right)\\
+\frac{k^\alpha}{\Gamma(\alpha+2)}
\sum_{j=0}^{n-1} a_j g(t_j,y^{pr}(t_j)), &  n\geq 1,
\end{array} \right.
\end{equation*}
where
\begin{equation*}
a_j= \left\{
\begin{array}{ll}
n^{\alpha+1}-(n-\alpha)(n+1)^\alpha, & j=0,\\
(n-j+2)^{\alpha+1}+(n-j)^{\alpha+1}-2(n-j+1)^{\alpha+1},
&  1\leq j \leq n,\\
1, & j=n+1
\end{array} \right.
\end{equation*}
and
\begin{equation*}
y^{pr}(t_{n+1})= \left\{
\begin{array}{ll}
y_0+\frac{k^\alpha}{\Gamma(\alpha+1)}g(t_0,y_0), & n=0,\\
y_0+\frac{k^\alpha}{\Gamma(\alpha+2)} \cdot (2^{\alpha h+1}-1) \cdot g(t_n,y(t_n))\\
+\frac{k^\alpha}{\Gamma(\alpha+2)}\sum_{j=0}^{n-1}a_j g(t_j,y(t_j)), &  n\geq1
\end{array} \right.
\end{equation*}
and $k$ is the step length, i.e., $k=t_{j+1}-t_j$ (for more details, see \cite{sal28}).
Finally, if $P_n$ is the interpolating polynomial passing through $\{(t_i,y1(t_i)\}_{i=0}^{n+1}$,
then we have the following formula for calculating $y(t_{n+1})$:
\begin{equation}
\label{eq6.12}
y(t_{n+1})\simeq\sum_{k=0}^{n-1}\frac{t_{n+1}^k}{k!}b_k
+\frac{1}{\Gamma(\alpha)}\left(\frac{t_{n+1}}{2}\right)^\alpha
\sum_{k=1}^N l_k g\left(\frac{t_{n+1}}{2}(1+x_k),
P_n\left(\frac{t_{n+1}}{2}(1+x_k)\right)\right).
\end{equation}

We solve two examples illustrating the applicability of our method.

\begin{example}
\label{exam5.3}
Consider the nonlinear ordinary differential equation
\begin{equation*}
{_0^CD_t ^{\alpha}}y(t)+y^2(t)=f(t)
\end{equation*}
subject to $y(0)=0$ and $y'(0)=0$, where
$$
f(t)=\frac{120t^{5-\alpha}}{\Gamma(6-\alpha)}-\frac{72t^{4-\alpha}}{\Gamma(5-\alpha)}
+\frac{12t^{3-\alpha}}{\Gamma(4-\alpha)}+k(t^5-3t^4+2t^3)^2.
$$
Following \cite{sal29}, we take $k=1$
and $\alpha=\frac{3}{2}$. The exact solution is then given by
$$
y(t)=(t^5-3t^4+2t^3)^2.
$$
Figure~\ref{Fig:6.1} plots the results obtained for $N=16$ in \eqref{eq6.12}.
\begin{figure}
\includegraphics[scale=.75]{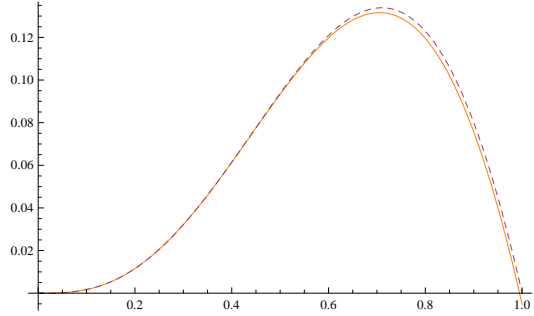}
\caption{The exact solution (the solid line) versus the approximated solution 
(the dashed line) of the fractional initial value problem of Example~\ref{exam5.3}.}
\label{Fig:6.1}
\end{figure}
\end{example}

\begin{example}
\label{exam5.4}
Consider the fractional oscillation equation
\begin{equation*}
{_0^CD_t ^{\alpha}}y(t)+y(t)=t e ^{-t}
\end{equation*}
subject to initial conditions  $y(0)=0$ and $y'(0)=0$. The exact solution is
$$
y=\int_0^t G(t-x)xe^{-x}dx,
\quad G(t)=t^{\alpha-1} E_{\alpha,\alpha}(t),
$$
where $E_{\alpha,\beta}(t)$ is the generalized Mittag--Lefller 
function \cite{sal29}. Results are shown in Figure~\ref{Fig:6.2} 
for $N=16$ in \eqref{eq6.12} and $\alpha=\frac{3}{2}$.
\begin{figure}
\includegraphics[scale=.75]{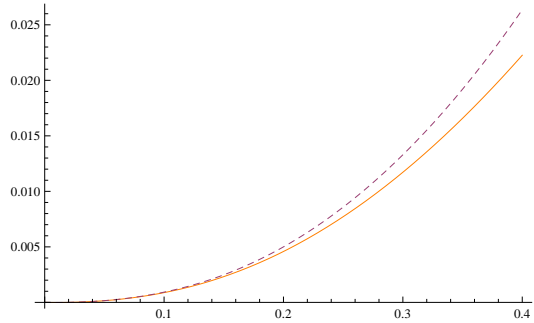}
\caption{The exact solution (the solid line) versus the approximated solution 
(the dashed line) of the fractional initial value problem of Example~\ref{exam5.4}.}
\label{Fig:6.2}
\end{figure}
\end{example}


\section{Application to fractional variational problems}
\label{sec6}

In this section we apply our method to solve, numerically, 
some problems of the calculus of variations of fractional order.
The calculus of variations is a rich branch of classical mathematics 
dealing with the optimization of physical quantities 
(such as time, area, or distance). It has applications in many
diverse fields, including aeronautics (maximizing the lift of an aircraft wing), 
sporting equipment design (minimizing air resistance on a bicycle helmet 
or optimizing the shape of a ski), mechanical engineering 
(maximizing the strength of a column, a dam, or an arch), 
boat design (optimizing the shape of a boat hull) and physics (calculating trajectories
and geodesics, in both classical mechanics and general relativity) \cite{sal32}.
The problem of the calculus of variations of fractional order is more recent
(see \cite{sal34,sal35,sal33} and references therein) and consists, typically,
to find a function $y$ that is a minimizer of a functional
\begin{equation}
\label{eq7.1}
J[y]=\int_{a}^{b}L\left(t,y(t),(_0^CD_t ^{\alpha})y(t)\right)dt,
\quad \alpha\in(n-1,n),
\quad n\in \mathbb{N},
\end{equation}
subject to boundary conditions
$$
y(a)=u_a,\quad y(b)=u_b.
$$
One can deduce fractional necessary optimality equations to problem \eqref{eq7.1}
of Euler--Lagrange type: if function $y$ is a solution to problem \eqref{eq7.1},
then it satisfies the fractional Euler--Lagrange differential equation
\begin{equation}
\label{Euler}
\frac{\partial L}{\partial y} + {_t D_b^\alpha}
\frac{\partial L}{\partial (_0^CD_t ^{\alpha}) }=0.
\end{equation}
It is often hard to find the analytic solution to problems of the calculus of variations
of fractional order by solving \eqref{Euler}. Therefore, 
it is natural to use some numerical method to find approximations 
of the solution of \eqref{Euler}. Here we are going to apply our method
for solving some examples of such fractional variational problems. To this end,
at first we find a class of fractional Lagrangians
\begin{equation*}
L\left(t,y(t),(_0^CD_t ^{\alpha})y(t)\right),
\quad 1<\alpha <2,
\end{equation*}
such that the corresponding Euler--Lagrange equation \eqref{Euler} takes the form
\begin{equation}
\label{eq7.2}
\frac{\partial L}{\partial y}+ {_t D_b^\alpha} \frac{\partial L}{\partial (_0^CD_t ^{\alpha})}
={_t D_b^\alpha}\left((_0^CD_t ^{\alpha})y(t)\right)+g(t,y(t)) = 0.
\end{equation}
We assume a solution to this problem as follows:
\begin{equation}
\label{eq7.3}
L\left(t,y(t),(_0^CD_t ^{\alpha})y(t)\right)
=\frac{1}{2}\left((_0^CD_t ^{\alpha})y(t)\right)^2+f(t,y(t)).
\end{equation}
Then, we evaluate the left-hand side of \eqref{eq7.2} for \eqref{eq7.3} and we obtain
\begin{equation*}
{_t D_b^\alpha}\left((_0^CD_t ^{\alpha})y(t)\right)
+\frac{\partial f(t,y(t))}{\partial y} = 0.
\end{equation*}
So, if the Lagrangian of the variational problem \eqref{eq7.1} is of form \eqref{eq7.3},
then we just need to solve the following boundary value problem of fractional order:
\begin{equation}
\label{eq7.4}
{_t D_b^\alpha}\left((_a^CD_t ^{\alpha})y(t)\right)+g(t,y(t))=0
\end{equation}
subject to the boundary conditions $y(a)=u_a$ and $ y(b)=u_b$, where
\begin{equation}
\label{eq:m:g}
g(t,y(t))=\frac{\partial f(t,y(t))}{\partial y}.
\end{equation}
Now, we apply the operator ${_tI_b^\alpha}$ to both sides of  equation \eqref{eq7.4}.
We get the following boundary value problem of fractional order:
\begin{equation}
\label{eq7.5}
\begin{cases}
(_a^CD_t^\alpha) y(t)=-{_tI_b^\alpha}g(t,y(t)), & 1<\alpha<2, \quad a<t<b,\\
y(a)=u_a, \quad y(b)=u_b.
\end{cases}
\end{equation}

\begin{theorem}
\label{theorem6.1}
Consider the Lagrangian \eqref{eq7.3}
be such that $g(t,y(t))$ given by \eqref{eq:m:g}
is continuous and let  $-{_tI_b^\alpha}g(t,y(t))=h(t,y(t))$.
Then the function $h: [a,b]\times \mathbb{R}\to \mathbb{R}$
is continuous and the problem \eqref{eq7.5}
is equivalent to the integral equation
\begin{multline}
\label{eq7.6}
y(t)=u_a+(u_b-u_a)t+\frac{1}{\Gamma(\alpha)}
\int_0^t (t-s)^{\alpha-1} h(s,y(s))ds\\
-\frac{t}{\Gamma(\alpha)}\int_0^1 (1-s)^{\alpha-1}h(s,y(s))ds.
\end{multline}
\end{theorem}

\begin{proof}
The continuity of $h$ is one of the properties of the Riemann--Liuoville
fractional integral operator ${_tI_b^\alpha}$ (see, for example, 
\cite{MR2218073,sal15}). The second part of the theorem 
is the main result of \cite{sal30}.
\end{proof}

Now, let $t_j=jk$, $j=0,1,\ldots,n$, and $k=\frac{b-a}{n}$.
Furthermore, let the numerical values of $y(t)$ at $t_0,t_1,\ldots,t_n$
have been obtained and let these values be $y_0,y_1,\ldots,y_n$, respectively,
and $y_0=u_a$. Now we are going to compute the value of $y(t)$ at $t_{n+1}$,
i.e., $y_{n+1}$. To this end, first we use the same idea of \cite{sal31}
to predict the values of $y_0,y_1,\ldots,y_n$. 
Then, we correct them using our method. Let
\begin{equation*}
a_j= \left\{
\begin{array}{ll}
\frac{k^\alpha}{\alpha(\alpha+1)}(n^{\alpha+1}-(n-\alpha)(n+1)^\alpha),
& j=0,\\
\frac{k^\alpha}{\alpha(\alpha+1)}((n-j+2)^{\alpha+1}
+(n-j)^{\alpha+1}-2(n-j+1)^{\alpha+1}),
&  1\leq j \leq n,\\
\frac{k^\alpha}{\alpha(\alpha+1)},
& j=n+1
\end{array} \right.
\end{equation*}
and
\begin{equation*}
b_j=\frac{k^\alpha}{\alpha}((n+1-j)^\alpha-(n-j)^\alpha.
\end{equation*}
Then, using \eqref{eq7.6}, we have
\begin{multline*}
y(t_{n+1})=u_a+(u_b-u_a)t_{n+1}
+\frac{1}{\Gamma(\alpha)}\int_0^{t_{n+1}}
(t_{n+1}-s)^{\alpha-1} h(s,y(s))ds\\
-\frac{t_{n+1}}{\Gamma(\alpha)}
\int_0^1 (1-s)^{\alpha-1}h(s,y(s))ds.
\end{multline*}
Therefore,
\begin{equation*}
y_{n+1}=u_a+(u_b-u_a)t_{n+1}+\frac{1}{\Gamma(\alpha)}
\sum_{i=0}^n a_i h(t_i,y_i)-\frac{t_{n+1}}{\Gamma(\alpha)}
\sum_{i=0}^n b_i h(t_i,y_i).
\end{equation*}
Finally, if $P_n$ is the interpolation polynomial of the nodes
$\{(t_i,y_i)\}_{i=0}^{n+1}$, then the following correction formula
is an approximation of the solution to problem \eqref{eq7.5} at the point $t_{n+1}$:
\begin{multline*}
y(t_{n+1})\simeq u_a+(u_b-u_a)t_{n+1}\\
+\frac{1}{\Gamma(\alpha)}\left(\frac{t_{n+1}}{2}\right)^\alpha
\sum_{k=1}^N l_k g\left(\frac{t_{n+1}}{2}(1+x_k),P_n(\frac{t_{n+1}}{2}(1+x_k)\right)
-\frac{t_{n+1}}{\Gamma(\alpha)}\sum_{i=0}^{n+1} b_i h(t_i,y_i),
\end{multline*}
where $x_k$ and $l_k$, $k=1,2,\ldots,N$, are the nodes and the weights
of the fractional Gauss--Jacobi quadrature rule. Now we are going to apply
this method to solve a concrete fractional problem of the calculus of the variations.

\begin{example}
\label{example6.2}
Consider the following Fractional Variational Problem (FVP):
to find the minimizer of the functional
\begin{equation}
\label{eq:fvp:a}
J[y]=\int_0^1 \left((_0^CD_t ^{\alpha}y(t))^2+y^2(t)\right)dt,
\quad 1 < \alpha <2,
\end{equation}
subject to
\begin{equation}
\label{eq:fvp:bc}
y(0)=1,\quad y(1)=\frac{2 e}{1 + e^2}.
\end{equation}
When $\alpha \rightarrow 1$, this FVP tends to the classical problem
of the calculus of variations
\begin{equation}
\label{eq:p:a1}
J[y]=\int_0^1 \left(y'(t)^2+y^2(t)\right)dt \longrightarrow \min,
\quad y(0)=1, \quad y(1)=\frac{2 e}{1 + e^2},
\end{equation}
which has the solution
\begin{equation}
\label{eq:s:a1}
y(t)=\frac{e^t + e^{2-t}}{1+e^2}.
\end{equation}
On the other hand, the fractional Euler--Lagrange equation
associated with \eqref{eq:fvp:a} is
\begin{equation}
\label{eq7.8}
2 \, {_t D_1^\alpha}\left(_0^CD_t ^{\alpha}y(t)\right)+2y(t)=0.
\end{equation}
Looking to Figure~\ref{Fig:3} and Figure~\ref{Fig:5},
we can see that the numerical solutions to the fractional Euler--Lagrange
equation \eqref{eq7.8} subject to boundary conditions \eqref{eq:fvp:bc}
are approximating the solution \eqref{eq:s:a1} to the classical variational
problem \eqref{eq:p:a1}. Note that the exact solution to the FVP
\eqref{eq:fvp:a}--\eqref{eq:fvp:bc} is unknown.
\begin{figure}
\includegraphics[scale=.75]{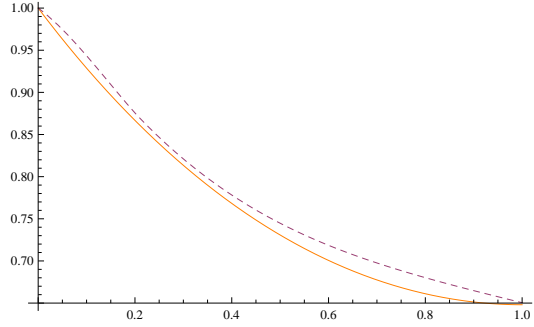}
\caption{The exact solution to \eqref{eq:p:a1} (solid line)
and the approximated solution to the FVP \eqref{eq:fvp:a}--\eqref{eq:fvp:bc}
with $\alpha=1.5$ (dashed line).}
\label{Fig:3}
\end{figure}
\begin{figure}
\includegraphics[scale=.75]{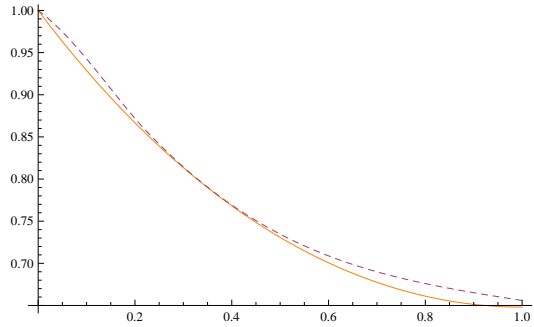}
\caption{The exact solution to \eqref{eq:p:a1} (solid line)
and the approximated solution to the FVP \eqref{eq:fvp:a}--\eqref{eq:fvp:bc}
with $\alpha=1.1$ (dashed line).}
\label{Fig:5}
\end{figure}
\end{example}


\section{Conclusions}
\label{sec7}

We presented a new numerical method for approximating Riemann--Liouvile
fractional integrals and Liouville--Caputo fractional derivatives. 
We applied the method for solving fractional linear and nonlinear 
initial value problems. Furthermore, we introduced a new method 
for solving fractional boundary value problems and applied this new method
for solving fractional Euler--Lagrange equations. Finally, by considering 
a concrete problem of the calculus of variations, we showed that the method 
is useful to solve fractional variational problems.

The results of the paper can be extended in several directions.
As an example of further possible developments in the field, 
we can mention the new fractional derivative with non-local 
and non-singular kernel proposed by Atangana and Baleanu \cite{Atangana:Bal}.
Such derivative seems to answer some outstanding questions that were posed 
within the field of fractional calculus, giving origin 
to chaotic behaviors \cite{MR3507052}. The importance to develop
new numerical methods to deal with such fractional calculus is well
presented in \cite{Alkahtani}. We claim that our results can be generalized
to cover the Atangana--Baleanu calculus. Similarly can be said
about numerical simulations with the Caputo--Fabrizio fractional 
order derivative \cite{MR3427809,Cap:Fab}.


\section*{Acknowledgements}

This work is part of first author's PhD project.
Partially supported by Islamic Azad University
(Science and Research branch of Tehran), Iran;
and CIDMA--FCT, Portugal, within project UID/MAT/04106/2013.
Jahanshahi was supported by a scholarship from the Ministry of Science,
Research and Technology of the Islamic Republic of Iran,
to visit the University of Aveiro, Portugal.
The hospitality and the excellent working conditions
at the University of Aveiro are here gratefully acknowledged.
The authors are grateful to two anonymous referees, for several comments 
and suggestions, which helped them to improve the manuscript. 



\bigskip


\end{document}